\documentclass[a4paper,12pt]{amsart}

\usepackage{float}
\usepackage{euscript,eufrak,verbatim}
\usepackage{graphicx}
\usepackage[usenames]{color}
\usepackage[colorlinks,linkcolor=red,anchorcolor=blue,citecolor=blue]{hyperref}
\usepackage{bm}
\usepackage{amsmath}
\usepackage{amsthm}
\usepackage[all]{xy}
\usepackage{graphicx}
\usepackage{amssymb}
\usepackage{enumerate}

\newtheorem{thm}{Theorem}[section]

\newtheorem{lem}[thm]{Lemma}
\newtheorem{cor}[thm]{Corollary}

\def\N{\mathbb{N}}

\def\N{\mathbb{N}}
\def\R{\mathbb{R}}

\newcommand{\ii}{\mathbf{i}}

\newcommand{\jj}{\mathbf{j}}
\newcommand{\zz}{\mathbf{z}}
\newcommand{\kk}{\mathbf{k}}
\newcommand{\llll}{\mathbf{l}}
\newcommand{\sss}{\mathbf{s}}

\makeatletter 
\@addtoreset{equation}{section}
\numberwithin{equation}{section}

\title[Spectrum of weighted Birkhoff average]{Spectrum of weighted	 Birkhoff average}

\author{Bal\'azs B\'ar\'any$^{1,2}$}
\address{$^1$ Department of Stochastics, Institute of Mathematics, Budapest University of Technology and Economics,	M\H{u}egyetem rkp. 3., H-1111 Budapest, Hungary}
\address{$^2$ MTA-BME Stochastics Research Group, M\H{u}egyetem rkp. 3., H-1111 Budapest, Hungary}
\email{balubsheep@gmail.com}

\author{Micha\l\ Rams$^3$}
\address{$^3$ Institute of Mathematics, Polish Academy of Sciences, ul. \'Sniadeckich 8, 00-656 Warszawa, Poland}
\email{rams@impan.pl}

\author{Ruxi Shi$^4$}
\address{$^4$ Sorbonne Universit\'e, LPSM, 75005 Paris, France}
\email{ruxi.shi@upmc.fr}

\thanks{{We would like to thank the anonymous referee her/his valuable comments and suggestions, which increased the level of presentation significantly.} Bal\'azs B\'ar\'any acknowledges support from grants OTKA K123782 and OTKA FK134251. Micha\l\ Rams was supported by National Science Centre grant 2019/33/B/ST1/00275 (Poland).}
\begin{document}
	
	\maketitle

\begin{abstract}
	Let $\{s_n\}_{n\in\N}$ be a decreasing nonsummable sequence of positive reals. In this paper, we investigate the weighted Birkhoff average $\frac{1}{S_n}\sum_{k=0}^{n-1}s_k\phi(T^kx)$ on aperiodic irreducible subshift of finite type $\Sigma_{\bf A}$ where $\phi: \Sigma_{\bf A}\mapsto \R$ is a continuous potential. Firstly, we show the entropy spectrum of the weighed Birkhoff averages remains the same as that of the Birkhoff averages. Then we calculate the packing spectrum of the weighed Birkhoff averages. { It turns out that we can have two cases, either the packing dimension of every level set equals to its Hausdorff dimension or for every nonempty level set it is equal to the packing dimension of the whole space.}
\end{abstract}

\section{Introduction}
Let $(X, d)$ be a compact metric space and $T : X \mapsto X$ be a continuous transformation.
 Let $\phi\colon X\mapsto\R$ be a continuous potential. The Birkhoff
 average of $\phi$ is given by
 $$
 \frac1n\sum_{k=0}^{n-1}\phi(T^kx).
 $$
 The Birkhoff Ergodic Theorem asserts that the Birkhoff averages almost surely converge to a constant, namely, the integral of $\phi$, with respect to any ergodic measure, for an integrable potential $\phi$. For a dynamical system having a large family of ergodic measures (for example, a full shift), one may expect that the limit of the Birkhoff averages can take a wide variety of values. Then one naturally wonders how large the size (for example, { topological entropy,} Hausdorff- and Tricot dimension) of the level sets of the limit of the Birkhoff averages is. In other words, for an $\alpha\in\R$ in the set
 $$
 L(\phi):=\left\{\alpha\in\R:\text{ there exists }x\in X\text{ such that }\lim_{n\to\infty}\frac1n\sum_{k=0}^{n-1}\phi(T^kx)=\alpha\right\}.
 $$
 one wants to know how to describe the size of
$$
E_\phi(\alpha):=\left\{x\in X:\lim_{n\to\infty}\frac1n\sum_{k=0}^{n-1}\phi(T^kx)=\alpha\right\}.
$$
It leads to the multifractal analysis and there has been a considerable amount of works on this. As far as we know, the first work is due to Besicovitch \cite{B1935} where he studied the Hausdorff dimension of sets given by the frequency of digits in dyadic expansions. Then it was
subsequently extended by Eggleston \cite{E1949}. For further results on digit frequencies, see Barreira, Saussol and Schmeling \cite{BSS}. For multifractal analysis of Birkhoff averages, we refer to \cite{O1998,FLW2002,S1999} and references therein.

Let $\{s_n\}_{n\in\N}$ be a monotone decreasing sequence of positive reals such that the series $S_n:=\sum_{k=0}^{n-1}s_k$ diverges, that is, $\lim_{n\to\infty}S_n=\infty$. In this paper, we are interested in the weighted Birkhoff average
$$
\lim_{n\to\infty} \frac{1}{S_n}\sum_{k=0}^{n-1}s_k\phi(T^kx).
$$
When $s_n=1$ for all $n$, it is the Birkhoff average of $\phi$. When $s_n=\frac{1}{n+1}$, it becomes $\frac{1}{\log n}\sum_{k=0}^{n-1}\frac{\phi(T^kx)}{k+1}$ which is sometimes called logarithmic Birkhoff average (because $\sum_{k=0}^{n-1}\frac{1}{k+1} \sim \log n$ when $n$ is large enough). The study of the logarithmic Birkhoff average has a link to logarithmic Sarnak conjecture and some problems in number theory. See more details in \cite{T2017} and references therein.

{ Let us note here that there also exists a different object that can be called a weighted Birkhoff average:
$$
\lim_{n\to\infty} \frac{1}{n}\sum_{k=0}^{n-1}s_k\phi(T^kx),
$$
that is the same kind of sums is considered with a different normalization. Those kinds of generalized Birkhoff averages were studied lately in \cite{F} and \cite{BRS}.
 }

By summation by parts, it is not hard to check (see also Lemma \ref{lem:trivialbound}) that the convergence of the Birkhoff average $\frac1n\sum_{k=0}^{n-1}\phi(T^kx)$ implies the convergence of the weighted Birkhoff average $\frac{1}{S_n}\sum_{k=0}^{n-1}s_k\phi(T^kx)$ and the limits are the same. The converse is not true in general, but it will turn out to be true under some condition. 

For $\alpha\in\R$ let
$$
E_\phi^s(\alpha):=\left\{x\in X:\lim_{n\to\infty}\frac{1}{S_n}\sum_{k=0}^{n-1}s_k\phi(T^kx)=\alpha\right\},
$$
In this paper, we investigate Hausdorff dimension and Tricot dimension of the level sets $E_\phi^s$ of the limit of the weighted Birkhoff averages. { In the literature, the Tricot dimension  is commonly called the packing dimension. However, since it was introduced by Claude Tricot \cite{T1982} we think that it is more appropriate appellation. Throughout the paper let us denote the Hausdorff- or covering dimension by $\dim_{\rm H}$, the Tricot- or packing dimension by $\dim_{\rm P}$ and the topological entropy by $h_{\rm top}$. For the definitions and basic properties, we direct the reader to \cite{PU}.}

\subsection{Main results}

A sequence $\{s_n\}_{n\in \N}$ of positive numbers is said to be {\it nonsummable} if $S_n=\sum_{k=0}^{n-1} s_{k} \to \infty$ as $n\to \infty$. For a decreasing nonsummable sequence $\{s_n\}_{n\in \N}$ of positive numbers,  we say that it has {\it bounded asymptotic ratio} if
\begin{equation}\label{eq:bar}
\limsup_{n\to \infty}\frac{S_n}{ns_{n}}<\infty.
\end{equation} Otherwise, it is said to have {\it unbounded asymptotic ratio}.

Notice that if $\lim_{n\to \infty}s_n=\alpha>0$, then it always has bounded asymptotic ratio.


\begin{thm}\label{thm:average}
	If $\{s_n\}_{n\in \N}$ are positive, nonincreasing, and nonsummable weights, and $\{a_n\}_{n\in \N}$ is any bounded sequence of reals, then the following is true:
	\begin{enumerate}[{\bf (i)}]
		\item\label{it:thmaver1} if $\limsup_{n\to \infty} \frac {S_{n}} {ns_n} < \infty$ then the sequence $\frac{1}{n} \sum_{k=0}^{n-1} a_k$ is convergent if and only if the sequence $\frac 1 {S_n} \sum_{k=0}^{n-1} s_k a_k$ is as well, and the limits are equal;
		\item\label{it:thmaver2} if $\limsup_{n\to \infty} \frac {S_{n}} {ns_n} = \infty$ then the convergence of the sequence $\frac{1}{n}\sum_{k=0}^{n-1} a_k$ implies convergence of the sequence $\frac 1 {S_n} \sum_{k=0}^{n-1} s_k a_k$ to the same limit, but the opposite is not true. For any such $\{s_n\}_{n\in \N}$ one can find $\{a_n\}_{n\in \N}$ such that the latter sequence is convergent and the former is divergent.
	\end{enumerate}
\end{thm}

This theorem has a simple corollary for the weighted multifractal spectra.

\begin{cor}\label{cor:bounded}
	Let $(X, d)$ be a compact metric space and $T : X \mapsto X$ be a continuous transformation. Then for every continuous potential $\phi\colon X\mapsto\R$ and for every monotone decreasing nonsummable sequence $\{s_n\}_{n\in\N}$ with bounded asymptotic ratio, we have for every $\alpha\in\R$
	$$
	E_\phi(\alpha)=E_\phi^s(\alpha).
	$$
	In particular, $h_{\rm top}(E_\phi(\alpha))=h_{\rm top}(E_\phi^s(\alpha))$, { $\dim_{\rm H}E_\phi(\alpha)=\dim_{\rm H}E_\phi^s(\alpha)$} and $\dim_{\rm P}E_\phi(\alpha)=\dim_{\rm P}E_\phi^s(\alpha)$.
\end{cor}

Let $\mathcal{A}=\{1,\ldots,K\}$ be a finite alphabet, and let $\Sigma=\mathcal{A}^\N$ be the symbolic space. Let $\sigma$ be the left-shift operator on $\Sigma$. Now we are concentrated on   $\Sigma_{\mathbf{A}}\subseteq\Sigma$ an aperiodic and irreducible subshift of finite type.

Our first result on weighted Birkhoff average is concerned about the entropy spectrum. Namely, regardless of the boundedness or unboundedness of the asymptotic ratio, we show that the entropy spectrum of the weighed Birkhoff averages remains the same as that of the Birkhoff averages on aperiodic and irreducible subshifts of finite type.

\begin{thm}\label{thm:hausdorff}
	For every continuous potential $\phi\colon\Sigma_{\bf A}\mapsto\R$ and for every monotone decreasing sequence $\{s_n\}_{n\in\N}$ of positive reals such that $\lim_{n\to\infty}s_n=0$ and $\lim_{n\to\infty}S_n=\infty$ and every $\alpha\in\R$, we have
	$$
	h_{\rm top}(E_\phi(\alpha))=h_{\rm top}(E_\phi^s(\alpha)).
	$$
\end{thm}

{ For H\"older continuous potentials and when the symbolic space is a full shift this result is actually an easy consequence of \cite[Theorem 3.3]{FST2013}, to obtain the full generality we first need a more general formulation of Fan, Schmeling, Troubetzkoy Theorem.}

{ Note that if the sequence $\{s_n\}_{n\in\N}$ has bounded asymptotic ratio then
	$$
	h_{\rm top}(E_\phi^s(\alpha))=\dim_{\rm H}E_\phi^s(\alpha)=\dim_{\rm P}E_\phi^s(\alpha).
	$$
This follows by Corollary~\ref{cor:bounded}, and the fact that $h_{\rm top}(E_\phi(\alpha))=\dim_{\rm H}E_\phi(\alpha)=\dim_{\rm P}E_\phi(\alpha)$, see Fan, Feng and Wu \cite{FFW}.}

Our second main result on weighted Birkhoff average is concerned about Tricot dimension. We show that when varying the sequence $\{s_n\}_{n\in \N}$, the change of the packing spectrum of the weighed Birkhoff averages has a "gap" in the sense that it equals to either that of the Birkhoff averages or the whole space.
\begin{thm}\label{thm:packing}
	Let $\{s_n\}_{n\in\N}$ be a monotone decreasing sequence of positive reals such that 
	`$\lim_{n\to\infty}S_n=\infty$ and $\limsup_{n\to\infty}\frac{S_{n}}{ns_{n}}=\infty$. Then for every continuous potential $\phi\colon\Sigma_{\bf A}\mapsto\R$, we have $E_\phi^s(\alpha)=\emptyset$ for all $\alpha\in\R\setminus L(\phi)$ and for every $\alpha\in L(\phi)$, we have
	$$
	\dim_{\rm P}E_\phi^s(\alpha)=h_{\rm top}(\Sigma_{\bf A}).
	$$
\end{thm}

For example, the sequence $((n+1)^{-1})_{n\ge 0}$ satisfies the condition of Theorem \ref{thm:packing}; the sequence $(s_n=(n+1)^d)_{n\ge 0}, -1<d<0,$ satisfies the condition of Corollary~\ref{cor:bounded}.





\section{Properties of weighted averages}
The following lemma is well known and follows directly by summation by parts. We present the proof here for completeness.
\begin{lem}\label{lem:trivialbound}
	Let $\{a_n\}_{n\in \mathbb{N}}$ be a bounded sequence of real numbers. Let $\{s_n\}_{n\in \mathbb{N}}$ be a decreasing sequence of positive numbers such that $S_n=\sum_{k=0}^{n-1}s_k\to\infty$ as $n\to\infty$. Then
	\begin{equation}\label{eq:log average}
	\begin{split}
	\liminf_{n\to \infty} \frac{1}{n} \sum_{k=0}^{n-1}a_k
	&\le \liminf_{n\to \infty} \frac{1}{S_n} \sum_{k=0}^{n-1}s_ka_k \\
	&\le \limsup_{n\to \infty} \frac{1}{S_n} \sum_{k=0}^{n-1}s_ka_k \le \limsup_{n\to \infty} \frac{1}{n} \sum_{k=0}^{n-1}a_k.
	\end{split}		
	\end{equation}
\end{lem}

\begin{proof}
We only need to prove the last inequality of \eqref{eq:log average}. Indeed, the proof of the first inequality follows by considering $-a_k$ in place of $a_k$. Denote $A_n = \sum_{k=0}^{n-1} a_k$. By summation by parts, we have that
	\begin{equation*}
		\sum_{k=0}^{n-1}s_ka_k= \sum_{k=0}^{n-2} (s_{k}-s_{k+1})A_{k+1} + s_{n-1}A_n.
	\end{equation*}
Let $\alpha:=\limsup_{n\to \infty} \frac{1}{n} A_n$. For any $\varepsilon>0$ let $N$ be the integer such that $A_n\leq(\alpha+\varepsilon)n$ for every $n\geq N$. Thus,
	\[
	\begin{split}
\sum_{k=0}^{n-1}s_ka_k	&\leq C(N) +(\alpha+\varepsilon) \left(\sum_{k=0}^{n-2} k(s_{k}-s_{k+1}) + s_{n-1}n \right)\\
	&\leq C(N)+(\alpha+\varepsilon)S_n,
\end{split}
\]
where $C(N)$ is a constant depending on $N$ but not on $n\ge N$.
Since $S_n\to\infty$ as $n\to \infty$ and $\varepsilon>0$ was arbitrary we get the last inequality.
\end{proof}

The next lemma shows that if $s_n$ has bounded asymptotic ratio then the lemma above can be reversed.

\begin{lem}\label{lem:BAR}
	Let $\{a_n\}_{n\in \mathbb{N}}$ be a bounded sequence of real numbers. Let $\{s_n\}_{n\in \mathbb{N}}$ be a decreasing summable sequence of positive numbers having bounded asymptotic ratio $G=\limsup_{n\to \infty}\frac{S_{n+1}}{(n+1)s_{n}} (\geq1)$. Then
	$$
	\limsup_{n\to\infty}\frac{1}{n}\sum_{k=0}^{n-1}a_k\leq G\limsup_{n\to\infty}\frac{1}{S_n}\sum_{k=0}^{n-1}s_ka_k+(1-G)\liminf_{n\to\infty}\frac{1}{S_n}\sum_{k=0}^{n-1}s_ka_k,
	$$
	and,
	$$
	\liminf_{n\to\infty}\frac{1}{n}\sum_{k=0}^{n-1}a_k\geq G\liminf_{n\to\infty}\frac{1}{S_n}\sum_{k=0}^{n-1}s_ka_k+(1-G)\limsup_{n\to\infty}\frac{1}{S_n}\sum_{k=0}^{n-1}s_ka_k.
	$$
\end{lem}

\begin{proof}
	We will show only the first inequality, the proof of the second is again obtained by replacing $(a_k)_k$ with $(-a_k)_k$. Denote
	$$
	\alpha:=\liminf_{n\to\infty}\frac{1}{S_n}\sum_{k=0}^{n-1}s_ka_k\text{ and }\beta:=\limsup_{n\to\infty}\frac{1}{S_n}\sum_{k=0}^{n-1}s_ka_k.
	$$
	Denote also $B_n = \sum_{k=0}^{n-1} s_k a_k$.
	For every $\varepsilon>0$ there exists $N\geq1$ such that for every $n\geq N$
	$$
	\alpha-\varepsilon<\frac{B_n}{S_n}<\beta+\varepsilon.
	$$
	By summation by parts, we have
	\[
	\sum_{k=0}^{n-1}a_k=\sum_{k=0}^{n-1}s_k^{-1}s_ka_k=\sum_{k=0}^{n-2}(s_k^{-1}-s_{k+1}^{-1})B_{k+1}+s_{n-1}^{-1}B_n.
	\]
	Since $s_n$ is monotone decreasing  for every $n\geq N$, we get
	\[
	\sum_{k=0}^{n-1}a_k\leq C(N)+\sum_{k=0}^{n-2}(\alpha-\varepsilon)S_{k+1}(s_k^{-1}-s_{k+1}^{-1})+s_{n-1}^{-1}(\beta+\varepsilon)S_n,
	\]
	where $C(N)$ is a constant depending on $N$ but not on $n\ge N$. 
	 Hence,
	\[
	\begin{split}
		\sum_{k=0}^{n-1}a_k&\leq C(N)+\sum_{k=0}^{n-2}(\alpha-\varepsilon)S_{k+1}(s_k^{-1}-s_{k+1}^{-1})+s_{n-1}^{-1}(\beta+\varepsilon)S_n\\
		&=C(N)+n\alpha+\varepsilon\left(2s_{n-1}^{-1}S_{n}-n\right)+(\beta-\alpha)s_{n-1}^{-1}S_n.
	\end{split}
	\]
	Since $\varepsilon>0$ was arbitrary, the claim follows by the definition of $G$.
\end{proof}

{ Note that the the quantity given in the definition of bounded asymptotic ratio \eqref{eq:bar} equals to $\limsup_{n\to\infty}\frac{S_{n+1}}{(n+1)s_{n}}$ by simple algebraic manipulations because $\lim_{n\to\infty}S_{n+1}/S_n=1$.}

The following is a simple consequence of Lemma~\ref{lem:BAR}.

\begin{cor}\label{cor:BAR}
	Let $\{a_n\}_{n\in \mathbb{N}}$ be a bounded sequence of real numbers. Let $\{s_n\}_{n\in \mathbb{N}}$ be a decreasing summable sequence of positive numbers having bounded asymptotic ratio. Then
	$$
	\lim_{n\to\infty}\frac{1}{S_n}\sum_{k=0}^{n-1}s_ka_k=\alpha\text{ if and only if }\lim_{n\to\infty}\frac{1}{n}\sum_{k=0}^{n-1}a_k=\alpha.
	$$
\end{cor}

Clearly, Theorem~\ref{thm:average}\eqref{it:thmaver1} follows immediately by Lemma~\ref{lem:trivialbound} and Corollary~\ref{cor:BAR}. To show Theorem~\ref{thm:average}\eqref{it:thmaver2}, we need the following lemma, which will be used later also in the proof of Theorem~\ref{thm:packing}.

\begin{lem}\label{lem:UBAR}
	Let $\{s_n\}_{n\in \mathbb{N}}$ be a decreasing nonsummable sequence of positive numbers having unbounded asymptotic ratio. Then there exist sequences $\{n_k\}_{k\in\N}$ and $\{m_k\}_{k\in\N}$ satisfying that $n_k\leq m_k< n_{k+1}$ and
	\begin{equation}\label{eq:mainneed}
	\lim_{k\to\infty}\frac{k}{S_{n_k}}=\lim_{k\to\infty}\frac{\sum\limits_{\ell=1}^k(S_{m_\ell}-S_{n_\ell})}{S_{m_k}}=\lim_{k\to\infty}\frac{n_k}{m_k}=0.
	\end{equation}
\end{lem}

\begin{proof}
	First, observe that it is enough to show that there exist sequences $n_k$ and $m_k$ such that $n_k\leq m_k< n_{k+1}$,
	\begin{equation}\label{eq:enough}
	\lim_{k\to\infty}\frac{k}{S_{n_k}}=\lim\limits_{k\to\infty}\frac{n_k}{m_k}=0\text{ and }\sum\limits_{k=1}^\infty\frac{S_{m_k}-S_{n_k}}{S_{m_k}}<\infty.
	\end{equation}
	The implication of \eqref{eq:enough} $\Rightarrow$ \eqref{eq:mainneed} follows by Kronecker lemma \footnote{Kronecker lemma: Let $\{b_n\}$ be a sequence of positive increasing numbers with $\lim\limits_{n\to \infty} b_n= \infty$. Let $\{x_n\}$ be a sequence of numbers such that $\sum_{n=1}^{\infty} x_n$ converges. Then
	$$
\lim\limits_{N\to \infty} \frac{1}{b_N} \sum_{n=1}^{N} b_nx_n=0.
$$
We apply Kronecker lemma for $b_k=S_{m_k}$ and $x_k=\frac{S_{m_k}-S_{n_k}}{S_{m_k}}$.
} (\cite[Page 390, Lemma 2]{S1996}).
	
	In order to show \eqref{eq:enough} we construct sequences $n_k\le m_k$ $(k\in \N)$ such that
	\begin{equation}\label{eq:need1}
	\lim\limits_{k\to\infty}\frac{S_{m_k}}{S_{n_k}}=1\text{ and }\lim\limits_{k\to\infty}\frac{m_k}{n_k}=\infty.
	\end{equation}
	If \eqref{eq:need1} holds then by passing to some subsequences $n_{k_\ell}$ and $m_{k_\ell}$, we will get $m_{k_{\ell-1}}< n_{k_{\ell}}$ and $\frac{\ell}{S_{n_{k_{\ell}}}}\leq2^{-\ell},\frac{S_{m_{k_\ell}}-S_{n_{k_\ell}}}{S_{m_{k_\ell}}}\leq2^{-\ell}$ and so \eqref{eq:enough} follows for the sequences $n_{k_\ell}$ and $m_{k_\ell}$.
	
	It remains to find sequences $\{n_k\}_{k\in \N}$ and $\{m_k\}_{k\in \N}$ satisfying \eqref{eq:need1}. Let $n_k$ be the subsequence for which $\lim\limits_{k\to\infty}\frac{S_{n_k}}{n_ks_{n_k-1}}=\infty$. For simplicity, let $M_k=\frac{S_{n_k}}{n_ks_{n_k-1}}$. Let $m_k$ be the smallest positive integer greater than or equal to $n_k$ such that
	$$
	S_{m_k+1}-S_{n_k}\geq S_{n_k}M_k^{-1/2}.
	$$
	Thus,
	$$
	1\leq\frac{S_{m_k}}{S_{n_k}}<1+M_k^{-1/2},
	$$
	and hence, the first claim of \eqref{eq:need1} holds since $M_k\to\infty$ as $k\to \infty$. On the other hand, by the monotonicity of $s_n$
	\[
	(m_k-n_k+1)s_{n_k}\geq S_{m_k+1}-S_{n_k}\geq S_{n_k} M_k^{-1/2}
	\]
	and so,
	$$
	\frac{m_k-n_k+1}{n_k}\geq\frac{S_{n_k}}{n_ks_{n_k-1}}M_k^{-1/2}=M_k^{1/2},
	$$
	which implies that $m_k/n_k\to\infty$ as $k\to\infty$. This completes the proof.
\end{proof}

\begin{proof}[Proof of Theorem~\ref{thm:average}\eqref{it:thmaver2}]
	Let $n_k\leq m_k<n_{k+1}$ as in Lemma~\ref{lem:UBAR} which satisfy \eqref{eq:enough} and \eqref{eq:need1}. Let us define the sequence $a_\ell$ as follows:
	$$
	a_\ell:=
	\begin{cases}
	-1 & \text{if } n_{2k}\leq\ell< m_{2k}\text{ for some $k\in\N$},\\
	1 & \text{if } n_{2k+1}\leq\ell< m_{2k+1}\text{ for some $k\in\N$},\\
	0 &\text{if }  m_{k}\leq\ell<n_{k+1}.
	\end{cases}
	$$
	Let $n\in\N$ be arbitrary and let $k$ be such that $n_k\leq n<n_{k+1}$. Then we have
	$$
	\left|\sum_{k=0}^{n-1}s_ka_k\right|\le 	\sum_{k=0}^{n-1}s_k\left|a_k\right|\le \sum_{\ell=0}^{k-1}(S_{m_{\ell}}-S_{n_\ell})+S_{\max\{n,m_k\}}-S_{n_k}.
	$$
	It follows that
	\[
	\left|\frac{1}{S_n}\sum_{k=0}^{n-1}s_ka_k\right|\leq\frac{\sum_{\ell=0}^{k-1}(S_{m_{\ell}}-S_{n_\ell})}{S_{m_{k-1}}}+\frac{S_{m_k}-S_{n_k}}{S_{n_k}}.
	\]
	Since $n_k\leq m_k<n_{k+1}$ satisfy \eqref{eq:enough} and \eqref{eq:need1}, the right hand side tends to zero as $n$ tends to infinity. On the other hand, we have
	$$
	\sum_{j=0}^{m_{2k+1}-1}a_j\geq m_{2k+1}-n_{2k+1}-m_{2k} \text{ and }
	\sum_{j=0}^{m_{2k}-1}a_j\leq-m_{2k}+n_{2k}+m_{2k-1}.
	$$
	Since $m_{\ell}<n_{\ell+1}\le m_{\ell+1}$ and $\lim\limits_{\ell\to \infty}\frac{n_{\ell}}{m_{\ell}}=1$, we conclude that
	\[
	\begin{split}
	\limsup_{n\to\infty}\frac{1}{n}\sum_{k=0}^{n}a_k&\geq\lim_{k\to\infty} \left( \frac{m_{2k+1}-n_{2k+1}}{m_{2k+1}}-\frac{m_{2k}}{m_{2k+1}}\right)=1\text{ and}\\
	\liminf_{n\to\infty}\frac{1}{n}\sum_{k=0}^{n}a_n&\leq\lim_{k\to\infty}\left( \frac{m_{2k}-n_{2k}}{m_{2k}}(-1)+\frac{m_{2k-1}}{m_{2k}}\right)=-1.
	\end{split}
	\]
\end{proof}

\section{Variational principle}\label{sec:Variational principle}

\subsection{Subshift of finite type}
Let $\mathcal{A}=\{1,\ldots,K\}$ be a finite alphabet, and let $\Sigma=\mathcal{A}^\N$.  Denote $\Sigma_n$ the set of $n$-length finite word. Moreover, denote $\Sigma_*$ the set of all finite prefixes of the infinite words in $\Sigma$. For an $\ii=(i_0,i_1,\ldots)\in\Sigma$ and $m>n\geq0$ let $\ii|_n^m=(i_n,\ldots,i_m)$ be the subword of $\ii$ between the positions $n$ and $m$, and for short denote by $\ii|_n$ the first $n$ element of $\ii$, i.e. $\ii|_n=\ii|_{0}^{n-1}$. For an $\ii\in\Sigma_*$, denote $|\ii|$ the length of $\ii$ and let $[\ii]$ denote the corresponding cylinder set, that is, $[\ii]:=\{\jj\in\Sigma:\jj|_{|\ii|}=\ii\}$. We use $l(\cdot)$ to denote the level of cylinder. Moreover, The space $\Sigma$ is metrizable with metric
\begin{equation}\label{eq:metric}
	d(\ii,\jj)=e^{-\min\{n\geq0:i_n\neq j_n\}}.
\end{equation}
For short, denote $\ii\wedge\jj=\min\{n\geq0:i_n\neq j_n\}$. { Note that $h_{\rm top}=\dim_{\rm H}$ for $(\Sigma,d)$ with respect to the left-shift operator $\sigma$.}

Let $\mathbf{A}$ be a $K\times K$ matrix with entries $0,1$. We say that the set $\Sigma_{\mathbf{A}}\subseteq\Sigma$ is {\it subshift of finite type} if
$$
\Sigma_{\mathbf{A}}=\{\ii=(i_0,i_1,\ldots)\in\mathcal{A}^\N:\mathbf{A}_{i_k,i_{k+1}}=1\text{ for every }k=0,1,\ldots\}.
$$
We call the matrix $\mathbf{A}$ the {\it adjacency matrix}. Let us denote the set of admissible words with length $n$ (i.e. $n$-length subwords of some element in $\Sigma_{\mathbf{A}}$) by $\Sigma_{\mathbf{A},n}$ and denote $\Sigma_{\mathbf{A},*}$ the set of all admissible words. Without loss of generality, we may assume that $\Sigma_{\mathbf{A},1}=\mathcal{A}$. Moreover, we say that $\Sigma_{\mathbf{A}}$ is {\it aperiodic and irreducible} if there exists $r\geq1$ such that every entry of $\mathbf{A}^r$ is strictly positive.

\subsection{Continuous potentials} Let $\phi$ be a continuous potential. Denote the Birkhoff average by
$$\chi_\phi(\ii) = \lim_{n\to\infty} \frac 1n \sum_{j=0}^{n-1} \phi(\sigma^j(\ii)).$$ Similarly, we write that $\chi_\phi^-$ is the liminf and $\chi_\phi^+$ is the limsup. Denote by $H_\phi(s) := h_{\rm top}(\left\{\ii\in\Sigma_{\bf A}:\chi_\phi(\ii)=s \right\})$.

In this section, we prove a variational principle of continuous potential. To do this we need to use a result on the multifractal spectrum of irregular sets. The first result of this type comes from Olsen (see for example \cite{O}), but we will use a variant of the following result of Fan, Schmeling, Troubetskoy in \cite[Theorem 3.3]{FST2013}.

\begin{thm} \label{thm:fst}
	Let $(\Sigma_{\mathbf{A}}, \sigma)$ be an irreducible and aperiodic subshift of finite type. Let $\varphi: \Sigma_{\mathbf{A}} \to \R$ be a H\"older continuous potential. For any $t\in \R$ we have
	
	\[
	h_{\rm top}(\left\{\chi_\varphi^-<t \right\})= h_{\rm top}(\left\{\chi_\varphi^+<t \right\}) = \sup_{s<t} H_\varphi(s),
	\]
	\[
	h_{\rm top}(\left\{\chi_\varphi^-\geq t \right\}) = h_{\rm top}(\left\{\chi_\varphi^+\geq t \right\}) = \sup_{s\geq t} H_\varphi(s).
	\]
\end{thm}

{ We use here the convention that $H_\varphi(t)$ is $0$ outside $L(\varphi)$ and so for any $t$ smaller than or equal to the left-endpoint of $L(\varphi)$, $\sup_{s<t}H_\varphi(s)=0$. In particular, $\{\chi_\varphi^-<t\}=\{\chi_\varphi^+<t\}=\emptyset$ for such $t$.
	
Note that by replacing $\varphi$ to $-\varphi$ the equalities
\[
h_{\rm top}(\left\{\chi_\varphi^-\leq t \right\})= h_{\rm top}(\left\{\chi_\varphi^+\leq t \right\}) = \sup_{s\leq t} H_\varphi(s),
\]
\[
h_{\rm top}(\left\{\chi_\varphi^-> t \right\}) = h_{\rm top}(\left\{\chi_\varphi^+> t \right\}) = \sup_{s> t} H_\varphi(s)
\]
hold as well.}

The theorem above has been stated in \cite{FST2013} only for full shift, however, the used smoothness conditions during the proof (see \cite[Theorem~3.2]{FST2013}) hold for subshift of finite type, see { Barreira, Saussol and Schmeling \cite[Section~4.1]{BSS2}, which follows from \cite[Example~4.2.6, Remark~4.3.4, Theorem~4.3.5]{K} and \cite[Theorem~1.22]{Bo}.}

Let $\phi$ be a continuous potential. We want to investigate the multifractal properties of $\phi$. The following theorem is folklore and may be known to some experts. However we were unable to find it in the literature. So we prove it ourselves.
\begin{thm} \label{thm:continuous potential}
	Let $(\Sigma_{\mathbf{A}}, \sigma)$ be an irreducible and aperiodic subshift of finite type. Let $\phi: \Sigma_{\mathbf{A}} \to \R$ be a  continuous potential. For any $t\in \R$ we have
	
	\[
	h_{\rm top}(\left\{\chi_\phi^-<t \right\}) = h_{\rm top}(\left\{\chi_\phi^+<t \right\}) = \sup_{s<t} H_\phi(s),
	\]
	\[
	h_{\rm top}(\left\{\chi_\phi^-\geq t \right\}) = h_{\rm top}(\left\{\chi_\phi^+\geq t \right\}) = \sup_{s\geq t} H_\phi(s).
	\]
\end{thm}

For any natural number $n$, let $\phi_n$ be the maximal potential smaller than $\phi$ but constant on $n$-th level cylinders, namely,

\[
\phi_n(x) = \min_{y\in [x{|_n}]} \phi(y).
\]
Clearly, every $\phi_n$ is a H\"older continuous potential. Also, there is a decreasing sequence $\{\varepsilon_n \}_{n\in \N}$ with $\varepsilon_n \searrow 0$ such that
\begin{equation} \label{eqn:compn}
0 \leq \phi-\phi_n \leq \varepsilon_n.
\end{equation}
Moreover, for $m>n$ we have $\phi_n \leq \phi_m \leq \phi$, and hence
\begin{equation} \label{eqn:compmn}
0 \leq \phi_m-\phi_n \leq \varepsilon_n.
\end{equation}

To simplify the notation let us write $H_n = H_{\phi_n}$, $\chi_n=\chi_{\phi_n}$, $\chi_n^-=\chi_{\phi_n}^-$ and $\chi_n^+=\chi_{\phi_n}^+$. Let us state the well known properties of the function $H_n(s)$. By Barreira, Saussol and Schmeling \cite[Theorem~2]{BSS2}, it is a smooth concave function, nonnegative, defined on some closed interval $L(\phi_n)=[\alpha_n^-, \alpha_n^+]$, with a single maximum at some point $\alpha_n$, and $H_n(\alpha_n) = h_{\rm top}(\Sigma_{\mathbf{A}})$.

{ By \eqref{eqn:compmn}, we have $\alpha_n^+\leq\alpha_m^+\leq\alpha_n^++\varepsilon_n$ (and similar for $\alpha^-$'s) whenever $m>n$. Thus, the limits $\alpha^-: = \lim_{n\to \infty} \alpha_n^-$ and $\alpha^+: = \lim_{n\to \infty} \alpha_n^+$ exist. Also, $\alpha^+_n\leq\alpha^+ \leq \alpha_n^++\varepsilon_n$, and the same for $\alpha^-$. Clearly, $L(\phi)=[\alpha^-,\alpha^+]$.}

Firstly, we study the convergence property of $H_n$.
\begin{lem}\label{lem:Hn}
	The functions $H_n$ pointwise converge to a limit function $H$ on $[\alpha^-,\alpha^+)$. Furthermore, $H$ is continuous and concave on $(\alpha^-,\alpha^+)$, and $H_n$ converges to $H$ uniformly on every compact subinterval in $(\alpha^-,\alpha^+)$.
\end{lem}
\begin{proof}{
	By \eqref{eqn:compn} we see that for any $\ii\in \Sigma_{\mathbf{A}}$ and $m>n$
	\begin{equation}\label{eqn:compn2}
	0 \leq \chi_m^-(\ii) - \chi_n^-(\ii) \leq \varepsilon_n\text{ and }0 \leq \chi^-(\ii) - \chi_n^-(\ii) \leq \varepsilon_n
	\end{equation}
	(the same holds for $\chi^+$).
	
	Let us define two maps $H^\ell_n(t):=\sup_{s\leq t}H_n(s)$ and $H^r_n(t):=\sup_{s>t}H_n(s)$.} { Please pay attention to the fact that those definitions are not symmetric: we use sharp inequality in $H^r$ but not in $H^\ell$.} { By concavity and Theorem~\ref{thm:fst},
	\[
	H_n^\ell(t) =\begin{cases} h_{\rm top}(\Sigma_{\mathbf{A}}) & \text{ for $t>\alpha_n$,}\\
		H_n(t) & \text{ for }t\leq\alpha_n.
		\end{cases}
	\]
	Similarly,
	\[
	H_n^r(t) =\begin{cases} 	H_n(t) & \text{ for $t\in(\alpha_n,\alpha^+_n)$,}\\
		h_{\rm top}(\Sigma_{\mathbf{A}}) & \text{ for }t\leq\alpha_n.
	\end{cases}
	\]
	It is easy to see that the map $H^\ell_n$ is concave and continuous on $[\alpha_n^-,\alpha_n^+]$ and $H^r_n$ is concave and continuous on $[\alpha_n^-,\alpha_n^+)$.
	
	Let $t\in[\alpha^-,\alpha^+)$ now be arbitrary but fixed. Since $\alpha_n^-\leq\alpha^-$ and $\alpha_n^+\to\alpha^+$ as $n\to\infty$, $t\in L(\phi_n)$ for every sufficiently large $n$.  Furthermore, by \eqref{eqn:compn2} for every $n<m$ sufficiently large
	\[
	\{\chi_m^-\leq t\}\subseteq\{\chi_n^-\leq t\}\text{ and }\{\chi_n^->t\}\subseteq\{\chi_m^->t\}.
	\]
	Hence, by Theorem~\ref{thm:fst}
	$$
	H_m^\ell(t)\leq H_n^\ell(t)\text{ and }H_m^r(t)\geq H_n^r(t)
	$$
	for every $m>n$ sufficiently large.	Since $0\leq H_n^\ell(t),H_n^r(t)\leq h_{\rm top}(\Sigma_{\mathbf{A}})$, the limits $H^\ell(t):=\lim_{n\to\infty}H_n^\ell(t)$ and $H^r(t):=\lim_{n\to\infty}H_n^r(t)$ exist for all $t\in[\alpha^-,\alpha^+)$.
	
	Since the maps $H_n^\ell$ and $H_n^r$ are concave, we get that $H^\ell$ and $H^r$ are concave on $[\alpha^-,\alpha^+)$, and in particular, continuous on $(\alpha^-,\alpha^+)$. By Dini's Theorem, $H^\ell_n$ and $H^r_n$ converges uniformly to $H^\ell$ and $H^r$ on any compact subinterval of $(\alpha^-,\alpha^+)$.
	
	Since $H_n(t)=\min\{H^\ell_n(t),H^r_n(t)\}$, the claim follows for the map $H(t):=\min\{H^\ell(t),H^r(t)\}$.}
\end{proof}

{ In particular, it follows from the proof that $\sup_{s\leq t}H(s)=H^\ell(t)$ and $\sup_{s> t}H(s)=H^r(t)$ for $t\in[\alpha^-,\alpha^+)$.}

We have the following immediate consequence of previous properties.

\begin{lem}\label{lem2} For every $t\in(\alpha^-,\alpha^+)$,
	\[
	h_{\rm top}(\left\{\chi^-<t \right\})= h_{\rm top}(\left\{\chi^+<t \right\})= \sup_{s<t} H(s),
	\]
	\[
	h_{\rm top}(\left\{\chi^-\leq t \right\})= h_{\rm top}(\left\{\chi^+\leq t \right\})= \sup_{s\leq t} H(s).
	\]
\end{lem}
\begin{proof}
	{ By \eqref{eqn:compn2},
	$$
		\left\{\ii\in\Sigma_{\bf A}:\chi_{n}^-(\ii)<t-\varepsilon_n\right\} \subseteq \left\{\ii\in\Sigma_{\bf A}:\chi^-(\ii)<t \right\} \subseteq \left\{\ii\in\Sigma_{\bf A}:\chi_{n}^-(\ii)<t \right\}
	$$
	Hence, it follows from Theorem \ref{thm:fst} that for every $t\in(\alpha^-,\alpha^+)$ and every sufficiently large $n$
	$$
	H_n^\ell(t-\varepsilon_n)\leq h_{\rm top}(\{\chi^-<t\})\leq H^\ell_n(t).
	$$
	Hence, for $t\in(\alpha^-,\alpha^+)$ the claim follows by Lemma \ref{lem:Hn}. The proof of the other cases are similar and we omit it.}
\end{proof}

The final stage of our investigation will be the following result:

\begin{lem}\label{lem3}
	For $t\in [\alpha^-, \alpha^+)$, we have $H_\phi(t)=H(t)$.
\end{lem}
\begin{proof}
	{ The upper bound is almost immediate. If $t\in(\alpha^-,\alpha^+)$ and $\chi_\phi(\ii)=t$ then by \eqref{eqn:compn2} $\chi_n(\ii) \leq \chi_\phi(\ii) \leq \chi_n(\ii) + \varepsilon_n$, and
	\[
	h_{\rm top}(\{\chi_\phi=t\})\leq h_{\rm top}(\left\{\chi_n \leq t \leq \chi_n + \varepsilon_n\right\})\leq\min\{H^\ell_n(t),H^r_n(t-\varepsilon_n)\}.
	\]
	By Lemma~\ref{lem:Hn}, the left-hand side converges to $H(t)$ as $n\to\infty$. On the other hand, if $t=\alpha^-$ then again by \eqref{eqn:compn2} $$h_{\rm top}(\{\chi_\phi=\alpha^-\})\leq h_{\rm top}(\left\{\chi_n \leq\alpha^- \right\})\leq H_n^\ell(\alpha^-)\to H^\ell(\alpha^-)=H(\alpha^-).$$
	
	Let now $t\in(\alpha^-,\alpha^+)$ be arbitrary but fixed.} For the lower bound we need to construct a large (in the sense of entropy) set of points for which $\chi_\phi =t$. We will do this by constructing an ergodic $\sigma$-invariant measure $\mu$ such that $\int \phi d\mu = t$, then the generic point for this measure belongs to $\{\chi_\phi=t\}$ and the topological entropy of those points equals the metric entropy of $\mu$.
	
Let us first remind the important property of irreducible subshifts of finite type: there exists a constant $K$ such that for any two admissible words $\ii,\jj$ there exists an admissible word $\kk, |\kk|=K$ such that the word $\ii\kk\jj$ is admissible. Given two admissible words $\ii,\jj, |\jj|>2K$ we will call their {\it admissible concatenation} the word $\ii\kk\jj$, where $\kk$ is just the connecting word making the word $\ii\kk\jj$ admissible. Such $\kk$ is in general not uniquely defined, so we can for example choose the first in the alphabetical order of all the fitting $\kk$'s. Similarly, for any finite or infinite sequence of admissible words of length larger than $2K$ we can define their admissible concatenation.
	
	Denote
	\[
	\rho_n = \frac 1n \sum_{i=1}^n \varepsilon_i,
	\]
	we have $\rho_n\searrow 0$.
	
	Fix some $\varepsilon < \min(t-\alpha^-, \alpha^+-t)/3$. Let $n$ be large enough that $\varepsilon > 2\rho_n$. For every $\delta>0$ we can find $N>n+4K||\phi||/\varepsilon$ large enough that the two following properties hold. First, there exist at least $e^{N(H_n(t+2\varepsilon)-\delta)}$ different words $(\ii_i)_i$ of length $N$ such that there exists $\jj_i\in[\ii_i]$ such that
	
	\begin{equation} \label{eqn:xi}
	\left|\frac 1N \sum_{j=0}^{N-1} \phi_n(\sigma^j(\jj_i)) - t - 2\varepsilon\right| < \varepsilon.
	\end{equation}
	Second, there exist at least $e^{N(H_n(t-2\varepsilon)-\delta)}$ different words $(\kk_j)_j$ of length $N$ such that there is $\llll_j\in[\kk_j]$ such that
	
	\begin{equation} \label{eqn:yi}
	\left|\frac 1N \sum_{j=0}^{N-1} \phi_n(\sigma^j(\llll_j)) - t + 2\varepsilon\right| < \varepsilon.
	\end{equation}
	
	Note that if $\jj_i$ satisfies \eqref{eqn:xi} then every $\jj\in [\jj_i|_N]=[\ii_i]$ satisfies
	\[
	\left|\frac 1N \sum_{j=0}^{N-1} \phi(\sigma^j(\jj)) - t - 2\varepsilon\right| < \varepsilon + \rho_N + \varepsilon_n < \varepsilon + 2\rho_n < 2\varepsilon,
	\]
	hence if we take any point $z$ obtained by an infinite admissible concatenation of the words $\gamma_i$, we will have $\chi_\phi^-(z), \chi_\phi^+(z) > t$. Indeed, the Birkhoff sum over each 'connecting' part of the admissible concatenation construction gives a correction $K||\phi||<N\varepsilon/4$, and so $\chi_\phi^-(z)>t+ \varepsilon - \rho_n - \varepsilon/4=t+\varepsilon/4$. Similarly, any infinite admissible concatenation of words $\kk_j$ gives upper and lower Birkhoff averages strictly smaller than $t$.
	
	For $p\in [0,1]$ consider the measure $\mu_N^p$ defined as a distribution of infinite admissible concatenations of words $(\ii_i)_i \cup (\kk_j)_j$ according to the following rule: at any position we have with probability $p$ one of the words $\ii_i$ (with equal probability each) and with probability $1-p$ one of the words $\kk_j$ (with equal probability each). This measure is $\sigma^N$-invariant, so we define a new measure
	
	\[
	\mu^p = \frac 1N \sum_{k=0}^{N-1} \sigma^k_* \mu_N^p,
	\]
	this measure is $\sigma$-invariant. As $\mu_N^p$ is $\sigma^N$-ergodic, $\mu^p$ is $\sigma$-ergodic. Moreover, $\int \phi d\mu^p$ is a continuous function of $p$, with
	
	\[
	\int \phi d\mu^1 < t < \int \phi d\mu^0.
	\]
	We can thus find some measure $\mu^{p(t)}$ for which $\int \phi d\mu^{p(t)}=t$. Finally, an easy calculation shows that for every $p$
	
	\[
	h(\mu^p) \geq (1-\frac KN)\min\{H_n(t-2\varepsilon), H_n(t+2\varepsilon)\} - \delta.
	\]
	{ Since all the generic points of $\mu^{p(t)}$ have Birkhoff average of $\phi$ equal to $t$, letting $n\to\infty$, we have
	$$h_{\rm top}(\{\chi_\phi=t\})\geq\min\{H(t-2\varepsilon), H(t+2\varepsilon)\} - \delta.$$
	Since $H$ is continuous at $t$ by Lemma~\ref{lem:Hn}, passing with $\varepsilon$ and $\delta$ to 0 we get the lower bound, and hence the result is proven.}

	{ Finally, we show that $H_\phi(\alpha^-)=H(\alpha^-)$. By \cite[Theorem~2.2]{BRS}, the map $t\mapsto H_\phi(t)$ is continuous on $[\alpha^-,\alpha^+]$.
	Furthermore, by Lemma~\ref{lem:Hn}, $H$ is concave on $[\alpha^-,\alpha^+)$ and so, lower semi-continuous at $\alpha^-$. Hence,
	$$
	H(\alpha^-)\leq\liminf_{t\to\alpha^-}H(t)=\liminf_{t\to\alpha^-}H_\phi(t)=H_\phi(\alpha^-).
	$$}
\end{proof}

\begin{proof}[Proof of Theorem \ref{thm:continuous potential}]
	{ It follows by Lemma \ref{lem:Hn}, Lemma \ref{lem2} and Lemma \ref{lem3} that
	\[\begin{split}
	h_{\rm top}(\{\chi^-<t\})&=h_{\rm top}(\{\chi^+<t\})=\sup_{s<t}H_\phi(s)\text{ and}\\
	h_{\rm top}(\{\chi^-\leq t\})&=h_{\rm top}(\{\chi^+\leq t\})=\sup_{s\leq t}H_\phi(s)
	\end{split}
	\]
	for every $t\in(\alpha^-,\alpha^+)$. For $t\geq\alpha^+$ and $t<\alpha^-$, the claim is straightforward. Finally, for $t=\alpha^-$ we have 	$\{\chi^-<t\}=\{\chi^+<t\}=\emptyset$ and since
	$$
	\{\chi_\phi=\alpha^-\}\subseteq\{\chi^+\leq \alpha^-\}\subseteq \{\chi^-\leq \alpha^-\}\subseteq\{\chi_n^-\leq \alpha^-\}
	$$
	the claim follows by Lemma \ref{lem:Hn} and Lemma \ref{lem3}. Finally, the remaining equalities follow by taking the potential $-\phi$. { Indeed, observe that
$$
\{\chi^+_{-\phi}\leq \alpha^-(-\phi)\}=\{\chi^-_{\phi}\geq \alpha^+(\phi)\}\text{ and }\{\chi^-_{-\phi}\leq \alpha^-(-\phi)\}=\{\chi^+_{\phi}\geq \alpha^+(\phi)\},
$$
and so
\[
\begin{split}
\{\chi_\phi=\alpha^+(\phi)\}&=\{\chi_{-\phi}=\alpha^-(-\phi)\}\subseteq\{\chi^+_{-\phi}\leq \alpha^-(-\phi)\}\\
&\subseteq \{\chi^-_{-\phi}\leq \alpha^-(-\phi)\}\subseteq\{\chi_{-\phi_n}^-\leq \alpha^-(-\phi)\}.
\end{split}
\]
By Lemma \ref{lem:Hn} $H_{-\phi_n}(\alpha^-(-\phi))\to H_{-\phi}(\alpha^-(-\phi))=H_\phi(\alpha^+(\phi))$, which completes the proof.}}
\end{proof}

\begin{proof}[Proof of Theorem \ref{thm:hausdorff}]
	By Lemma \ref{lem:trivialbound}, for $\alpha\in \R$, we have $E_\phi(\alpha)\subset E_\phi^s(\alpha)$. It follows that $h_{\rm top}(E_\phi(\alpha))\le h_{\rm top}(E_\phi^s(\alpha))$.
	
	On the other hand, by Lemma \ref{lem:trivialbound} again, we have
	$$
	E_\phi^s(\alpha)\subseteq \{\chi_\phi^-\le \alpha \}~\text{and}~E_\phi^s(\alpha)\subseteq \{\chi_\phi^+\ge \alpha \}.
	$$
	Then by Theorem~\ref{thm:continuous potential} we have
	$$
	h_{\rm top}(E_\phi^s(\alpha))\le \min \{ \sup_{s\ge \alpha} H_\phi(s), \sup_{s\le \alpha} H_\phi(s) \}=H_\phi(\alpha).
	$$
	where the last equality follows by the concavity and unique maximum of the function $H_\phi$, which completes the proof.
\end{proof}

\section{Packing spectrum with unbounded asymptotic ratio}

Before we prove Theorem~\ref{thm:packing}, we need the following technical lemma, which is a direct consequence of Lemma~\ref{lem:UBAR}

\begin{lem}\label{lem:maintopacking}
	Let $\phi\colon\Sigma_{\bf A}\mapsto\R$ be a continuous potential and let $\alpha\in L(\phi)$. Suppose that the sequence $\{s_n\}$ has unbounded asymptotic ratio, and let $\{n_k\}_{k\in\N}$ and $\{m_k\}_{k\in\N}$ be sequences as in Lemma~\ref{lem:UBAR}. Finally, let $\ii\in E_\phi^s(\alpha)$. Then for every $\jj\in\Sigma_{\bf A}$ such that $j_\ell=i_\ell$ for $m_k\leq\ell\leq n_{k+1}-1, k\in \N, $ we have $\jj\in E_\phi^s(\alpha)$.
\end{lem}

\begin{proof}
	
{	Let $\rho_n=\max_{\ii,\jj:|\ii\wedge\jj|=n}|\phi(\ii)-\phi(\jj)|$. Notice that $\rho_n$ is decreasing to $0$ as $n$ tends to $\infty$.

Heuristically, the idea of the proof is easy: the union of intervals $\bigcup_k [n_k, m_k]$ has positive upper frequency but has zero 'weighted' frequency. Thus, if we take some sequence $(a_\ell)_\ell$ and then change the values of some $a_\ell, \ell \in [n_k, m_k]$ then the weighted Birkhoff average will not change. Unfortunately, when $a_\ell = \phi(\sigma^\ell(\jj))$ this does not immediately work: the potential $\phi$ might depend on not only the first symbol, and thus modifying $\ii$ on positions $\ell \in [n_k, m_k]$ might also change $\phi(\sigma^\ell \ii)$ for other $\ell$'s. Fortunately, the potential $\phi$ is by assumption continuous, thus it has uniformly decreasing variations, and hence the change of $\phi(\sigma^\ell \ii)$ for $m_k \leq \ell < n_{k+1}$ is bounded by $\rho_{n_{k+1}-\ell}$.

 In particular,
\begin{equation}\label{eq:bound}
|\phi(\sigma^\ell\ii)-\phi(\sigma^\ell\jj)|\leq\begin{cases}
	\rho_0 & \text{if }n_k\leq\ell<m_k \\
	\rho_{n_{k+1}-\ell} & \text{if }m_k\leq \ell<n_{k+1}
\end{cases}
\end{equation}
for some $k$. This bound on the effect of modification of $\ii$ turns out to be sufficient to prove that the weighted Birkhoff average does not change.

Let $\varepsilon>0$ be arbitrary but fixed and let $N$ be such that $\rho_N\leq\varepsilon$. Let $k\in\N$ be such that $n_k\leq n\leq n_{k+1}-1$.  To prove the claim, we show that the sequence $\mathcal{I}_n:=\left|\frac{1}{S_n}\sum_{\ell=0}^{n-1}s_\ell\phi(\sigma^\ell\jj)-\frac{1}{S_n}\sum_{\ell=0}^{n-1}s_\ell\phi(\sigma^\ell\ii)\right|$
	converges to zero. By \eqref{eq:bound},
	\[
	\begin{split}
	\mathcal{I}_n&\leq\frac{\rho_0}{S_n}\sum_{\ell=1}^{k-1}(S_{m_\ell}-S_{n_\ell})+\frac{1}{S_n}\sum_{\ell=1}^{k-1}\sum_{p=m_\ell}^{n_{\ell+1}-1}s_p\rho_{n_{\ell+1}-p}+\frac{1}{S_n}\sum_{\ell=n_k}^{n-1}s_\ell|\phi(\sigma^\ell\ii)-\phi(\sigma^\ell\jj)|\\
	&=:\mathcal{I}_n^{(1)}+\mathcal{I}_n^{(2)}+\mathcal{I}_n^{(3)}.
	\end{split}
	\]
	Applying \eqref{eq:bound} again, if $n\leq m_k-1$ then
	$$
	\mathcal{I}_n^{(3)}\leq\frac{S_n-S_{n_k}}{S_n}\rho_0\leq \frac{S_{m_k}-S_{n_k}}{S_{m_k}}\rho_0,
	$$
	and if $m_k\leq n$ then
	\[
	\begin{split}
	\mathcal{I}_n^{(3)}&=\frac{S_{m_k}-S_{n_k}}{S_n}\rho_0+\frac{1}{S_n}\sum_{p=m_k}^{n-1}s_p\rho_{n_{k+1}-p}\leq\frac{S_{m_k}-S_{n_k}}{S_{m_k}}\rho_0+\frac{1}{S_n}\sum_{p=m_k}^{n-1}s_p\rho_{n_{k+1}-p}\\
	&\frac{S_{m_k}-S_{n_k}}{S_{m_k}}+\frac{Ns_0\rho_0}{S_n}+\frac{\varepsilon}{S_n}\sum_{p=m_k}^{n-N-1}s_p\leq\frac{S_{m_k}-S_{n_k}}{S_{m_k}}+\frac{Ns_0\rho_0}{S_n}+\varepsilon.
	\end{split}
	\]
	Hence, $\limsup_{n\to\infty}\mathcal{I}^{(3)}_n\leq\varepsilon$ by Lemma~\ref{lem:UBAR}.
	
	On the other hand,
	\[
	\begin{split}
	\mathcal{I}_n^{(2)}\leq \frac{\varepsilon}{S_n}\sum_{\ell=1}^{k-1}(S_{n_{\ell+1}-N}-S_{m_{\ell}})+\frac{\rho_0s_0(k-1)N}{S_n}\leq \varepsilon+\frac{\rho_0s_0(k-1)N}{S_n}.
	\end{split}\]
	By Lemma~\ref{lem:UBAR}, $\limsup_{n\to\infty}\mathcal{I}_n^{(2)}\leq\varepsilon$. Finally, applying Lemma~\ref{lem:UBAR} again
	$$
	\mathcal{I}^{(1)}_n\leq \frac{\rho_0}{S_{m_{k-1}}}\sum_{\ell=1}^{k-1}(S_{m_\ell}-S_{n_\ell})\to0
	$$
	as $n\to\infty$. Combining the above, $\limsup_{n\to\infty}\mathcal{I}_n\leq2\varepsilon$, and since $\varepsilon>0$ was arbitrary, the claim of the lemma follows.}
\end{proof}

\begin{proof}[Proof of Theorem~\ref{thm:packing}]
	Since $L(\phi)$ is an interval, for every $\alpha\notin L(\phi)$ either $\{\chi_\phi^-\leq\alpha\}$ or $\{\chi_\phi^+\geq\alpha\}$ is an empty set. {It follows by a straightforward adaptation of the proof of \cite[Theorem~3.1]{S1999}.}
	
	Thus, by Lemma~\ref{lem:trivialbound}, $E_\phi^s(\alpha)=\emptyset$ for $\alpha\notin L(\phi)$.
	
	Let $\alpha\in L(\phi)$ and let $\ii\in E^s_\phi(\alpha)$ be arbitrary but fixed. By \cite[Theorem 1]{T1982}, it is enough to show that there exists a probability measure $\mu$ such that $\mu(E_\phi^s(\alpha))=1$ and
	$$
	\limsup_{n\to\infty}\frac{1}{n}\log\mu([\ii|_n])=h_{\rm top}(\Sigma_{\bf A})\text{ for $\mu$-almost every }\ii.
	$$
	Let $\nu$ be the shift invariant ergodic measure on $\Sigma_{\bf A}$ for which $h_\nu=h_{\rm top}(\Sigma_{\bf A})$. { We remind that there exists a positive integer $r$ such that the $r$-th power of the adjacency matrix $A$ has strictly positive entries. In other words, for any two admissible words $\ii, \jj\in\Sigma_{\mathbf{A},*}$ there exist a word $\kk$ of length $r$ such that the word $\ii\kk\jj$ is admissible. Furthermore, $\kk$ depends only on the last symbol of $\ii$ and the first symbol of $\jj$. Let us define a map $\kk\colon\mathcal{A}\times\mathcal{A}\mapsto\Sigma_{\mathbf{A},r}$ such that $\ii\kk(i_{|\ii|},j_0)\jj$ is admissible in arbitrary but fixed way.} Let $\{n_k\}_{k\in\N}$ and $\{m_k\}_{k\in\N}$ be sequences as in Lemma~\ref{lem:UBAR}. {Let us define a subset of $F\subset\Sigma_{\mathbf{A}}$ as follows
	\[\begin{split}
	F=\{\jj\in\Sigma_{\mathbf{A}}:&\jj|_{m_{\ell-1}}^{n_\ell-1}=\ii|_{m_{\ell-1}}^{n_\ell-1},\ \jj|_{n_\ell}^{n_\ell+r-1}=\kk(i_{n_\ell-1},j_{n_\ell+r})\text{ and }\\
	&\jj|_{m_\ell-r+1}^{m_\ell-1}=\kk(j_{m_\ell-r},i_{m_\ell})\text{ for every }\ell\in\N\}
\end{split}
	\]
By Lemma~\ref{lem:maintopacking}, $F\subseteq E_\phi^s(\alpha)$. Let us define $\mu$ over cylinders $[\jj]$ of length $m_k$ with $[\jj]\cap F\neq\emptyset$ as follows:
	$$
	\mu([\jj])=\prod_{\ell=1}^{k}\nu([\jj|_{n_\ell+r}^{m_\ell-r}]).
	$$
	It is easy to see that $\mu$ is a well-defined probability measure supported on $F$ by
$$
\sum_{\jj\in\Sigma_{\mathbf{A},m_{k+1}-m_k}}\mu([\ii\jj])=\mu([\ii])\sum_{\jj\in\Sigma_{\mathbf{A},m_{k+1}-m_k-2r}}\nu([\jj])=\mu([\ii])
$$ for every $\ii\in\Sigma_{\mathbf{A},m_k}$.} On the other hand, for $\mu$-almost every $\jj$
\[
\begin{split}
	\limsup_{n\to\infty}\frac{-1}{n}\log\mu([\jj|_n])&\geq\limsup_{k\to\infty}\frac{-1}{m_k}\sum_{\ell=1}^{k}\log\nu([\jj|_{n_\ell+r}^{m_\ell-r}])\\
	&\geq\limsup_{k\to\infty}\frac{-1}{m_k}\log\nu([\jj|_{n_k+r}^{m_k-r}])\\
	&=h_\nu\limsup_{k\to\infty}\frac{m_k-n_k-2r}{m_k}=h_\nu=h_{\rm top}(\Sigma_{\bf A}),
\end{split}
\]
where the first equality follows by the Shannon-McMillan-Breiman Theorem and Lemma~\ref{lem:UBAR}, while the last equality follows by the choice of the measure $\nu$.
\end{proof}

\end{document}